\documentclass[fullpage]{amsart}

\usepackage{amssymb,xcolor,vaucanson-g}
\usepackage{hyperref} 
\usepackage{psfrag} 
\usepackage{epsfig}
\usepackage{amsmath, amsfonts, amstext, amscd, bezier}
\usepackage{pst-node}
\usepackage{graphicx}

\newcommand{\A}{{\mathcal A}}
\newcommand{\C}{{\mathcal C}}
\newcommand{\D}{{\mathcal D}}
\newcommand{\E}{{\mathcal E}}

\newcommand{\RR}{{\mathbb R}}
\newcommand{\CC}{{\mathbb C}}
\newcommand{\NN}{{\mathbb N}}
\newcommand{\ZZ}{{\mathbb Z}}

\newcommand{\Er}{{\mathcal E}_{1,1}}
\newcommand{\Eart}{{\mathcal E}_{1,1,1}}

\newcommand{\Egen}{{\mathcal E}_{a_1,\dots ,a_d}}

\newcommand{\Dgen}{{\mathcal D}_{a_1,\dots ,a_d}}
\newcommand{\Dgeninf}{{\mathcal D}_{a_1,\dots ,a_d}^{\infty}}
\newcommand{\Dinfty}{{\mathcal D}^{\infty}}

\newtheorem{theo}{Theorem}
\newtheorem{prop}[theo]{Proposition}

\newtheorem{coro}[theo]{Corollary}
\newtheorem{lemma}[theo]{Lemma}

\newcommand{\MinusScale}{0.35}
\newcommand{\MinusPicture}{\FixVCScale{\MinusScale}}

\begin{document}

\title{Boundary of the Rauzy fractal sets in $\RR \times \CC$ generated by $P(x) = x^4-x^3-x^2-x-1$}
\author{F. Durand}
\address{Laboratoire Ami\'enois
de Math\'ematiques Fondamentales et Appliqu\'ees, CNRS-UMR 6140,
Universit\'{e} de Picardie Jules Verne, 33 rue Saint Leu, 80039
Amiens Cedex, France.}

\email{fabien.durand@u-picardie.fr}

\author{A. Messaoudi}
\address{Departamento de Matem\'atica,
 Unesp-Universidade Estadual Paulista,
Rua Cristov\~ao Colombo, 2265, Jardim Nazareth, 15054-000, S\~ao
Jos\'e do Rio Preto, SP, Brasil.}

\email{messaoud@ibilce.unesp.br}

 \keywords{}

\pagestyle{plain}

\begin{abstract}
We study the boundary of the $3$-dimensional Rauzy fractal ${\mathcal E} \subset \RR \times \CC$ generated by the polynomial $P(x) = x^4-x^3-x^2-x-1$.
The finite automaton characterizing the boundary of ${\mathcal E}$ is given explicitly.
As a consequence we prove that the set ${\mathcal E}$ has $18$ neighborhoods
where $6$ of them intersect the central tile ${\mathcal E}$ in a point.
Our construction shows that the boundary is generated by an iterated function system
starting with $2$ compact sets.
\end{abstract}

\maketitle
\today

\section{Introduction}
\label{intro}

Consider $A= \{1,2,3\}$  as an alphabet. Let $A^{*}$ be the
set of finite words on $A$ and $\sigma : A\to A^{*}$ be the map
(called Tribonacci substitution) defined by
$$
\sigma (1) = 12,\; \sigma (2) = 13,\; \sigma (3) = 1.
$$
We extend $\sigma $ to $A^{\NN}$ by concatenation :
$\sigma(a_0\cdots a_{n}\ldots)=\sigma(a_0)\cdots \sigma (
a_{n})\ldots$. It is clear that $\sigma $ has a unique fixed point
$u$ : $\sigma (u) = u\in A^{\NN}$.  The dynamical system
associated to $\sigma$ is the couple $(\Omega, S)$ where $S :
A^{\NN} \to A^{\NN}$ is the shift map ($S ((x_n)_{n\in
\NN})=(x_{n+1})_{n\in \NN}$) and $\Omega$ is the $S$-orbit closure
of $u$ : $\Omega = \overline{\{ S^n u | n\in \NN\}}$.
It is well known that $(\Omega , S)$ is minimal, uniquely ergodic and of zero entropy (see \cite{Q87,F02}
for more details).

In 1982, G. Rauzy \cite{R82} studied the Tribonacci substitution
$\sigma$.
 He proved that the dynamical system generated by
$\sigma$ is measure theoretically conjugate to an exchange of domains $X_1, X_2, X_3$ in a compact tile
$X=X_1 \cup X_2 \cup X_3.$
The set $X$ is the classical two-dimensional Rauzy fractal.
It has been extensively studied and is related to many topics : numeration systems \cite{M00,M06,M05},
 geometrical representation of symbolic dynamical systems \cite{AI01,AIS01,CS01,HZ98,M98,T06,S96},
multidimensional continued fractions and simultaneous approximations \cite{ABI02,CHM01,C02,HM06}, self-similar tilings \cite{A99,A00,AI01,P99}
and Markov partitions of Hyperbolic automorphisms of the torus \cite{KV98,M98,P99}.

Among the main properties of the set $X$, let us recall it is compact,
connected, its interior is simply connected, its boundary is fractal
 and it induces a periodic tiling of $\mathbb{R}^2$ (\cite{R82}).

It is possible to associate such a fractal set to a large class of substitutions over an alphabet with $d$ letters  (called unimodular Pisot substitutions).
Let us call them {\it Rauzy fractals}.
P. Arnoux and S. Ito \cite{AI01} (see also \cite{CS01}) proved that the dynamical system associated to such a substitution $\sigma$
is measure theoretically conjugate to an exchange of domains $X_1, \ldots, X_d$ in the Rauzy fractal
$X_{\sigma}=X_1 \cup \ldots \cup X_d \subset \mathbb{R}^{d-1}$ provided that a certain combinatorial condition is true.
All these sets $X_{\sigma}$ are compact and generate a periodic tiling of
 $\mathbb{R}^{d-1}$.

There are different ways to define the Rauzy fractal associated to a given substitution $\sigma$ over an alphabet of $d+1$ letters.
One is through numeration systems.

Let $d \geq 2$ and $a_1, a_2, \ldots, a_d $ be integers such that
$a_1 \geq a_2 \geq \ldots \geq a_d \geq 1$.
Consider $A= \{1,2,\ldots, d+1\}$ as an alphabet.
Let  $\sigma_d$ be the substitution  defined by

$$
\sigma_d (i)=\underbrace{11\ldots 1}_{a_i}\;(i+1) \hbox { if } i \leq d
\hbox { and } \sigma_d (d+1)=1.
$$

We define the Rauzy fractal associated to $\sigma_d$ as follows.
Consider the sequence $(F_n)_{n \geq 0}$ defined by

$$
F_{n+d+1}= a_1 F_{n+d}+ a_2 F_{n+d-1}+\cdots + a_d F_{n+1}+F_n ,\; \forall n \geq 0 ,
$$

with initial conditions (called {\it Parry conditions})
$$F_0=1,\; F_{n}= a_1 F_{n-1}+ \cdots + a_n F_{0}+1, \; \forall \; 1 \leq n \leq d.$$

It is well-known (using the greedy algorithm) that for
every integer $n$ we have $n= \sum_{i=0}^{N} c_{i}F_i$
with $(c_{i})_{0 \leq i \leq N} \in \Dgen$, where $\Dgen $
is the set of sequences $(\varepsilon_{i})_{l \leq i \leq k}$, $l, k \in \ZZ$,
such that for all $ l\leq i \leq k$ :

\begin{enumerate}
\item
$\varepsilon_{i}\in \{ 0,1, \dots , a_1 \}$ ;
\item
$\varepsilon_{i}\varepsilon_{i-1} \ldots \varepsilon_{i-d}
 <_{lex} a_1 a_2 \ldots a_d 1$ when $ i \geq l+ d  $ .
\item
$\varepsilon_{i}\varepsilon_{i-1} \ldots \varepsilon_{l}0^{d-i+l}
 <_{lex} a_1 a_2 \ldots a_d 1$ when $ l \leq i \leq l+ d $,
\end{enumerate}

where $<_{lex}$ is the usual lexicographic ordering. We set

$$
\Dgeninf =
\left\{ (\varepsilon_i)_{i\geq l} ; l\in \ZZ ,
(\varepsilon_{i})_{l \leq i \leq n} \in \Dgen, \forall n\geq l
\right\} .
$$

Now, consider the following polynomial

$$
P_{a_1, \ldots a_d}(x)=x^{d+1} -a_1 x^d-a_2x^{d-1}- \cdots - a_d x-1 .
$$

It can be checked that $P$ has a root $\beta= \beta_1  \in
]1,+\infty [$ and $d$ roots with modulus less than $1$. Let
$\beta_1, \beta_2,\beta_3,\ldots, \beta_r$ be the roots of $P$
belonging to $\RR$ and $\beta_{r+1},\ldots, \beta_{r+s},
\overline{\beta_{r+1}},\ldots,\overline{\beta_{r+s}}$ its complex
roots.
For all $i \in \mathbb{Z}$, we set

$$
\alpha^i=
(\beta_{2}^{i},\ldots, \beta_{r}^{i},
\beta_{r+1}^{i},\ldots,\beta_{r+s}^{i}).
$$

We also put $\alpha^0=
1=(1,\ldots, 1).$
Then, the Rauzy fractal associated to $\sigma$ is the set ${\mathcal E}_{a_1, \ldots a_d} \subset
  \mathbb{R}^{r-1} \times \mathbb{C}^{s} \approx \mathbb{R}^{d}$
defined by

$$
{\mathcal E}_{a_1, \ldots a_d} = \left\{ \sum_{i=d+1}
^{+\infty}\varepsilon_{i}\alpha^{i} ;(\varepsilon_{i})_{i \geq d+1}
\in \Dgeninf
 \right\}.
 $$

The set ${\mathcal E}_{1,1}= X$ is the classical two-dimensional Rauzy fractal.

The structure of the boundary of Rauzy fractals has been first investigated by Ito and M. Kimura in \cite{IK91}.
They showed that the boundary of
${\mathcal E}_{1,1}$ is a Jordan curve generated by the Dekking
method \cite{D82} and they calculated its Hausdorff dimension.
Relating the boundary of ${\mathcal E}_{a_1,1}$
to the complex numbers having at least two expansions in base
$\alpha$,  A. Messaoudi \cite{M00,M05} constructed  a finite automaton characterizing and
generating this boundary.
As a consequence it permitted to
parameterize the boundary of ${\mathcal E}_{a_1,1}$, to compute its the Hausdorff dimension
 and to show it is a quasi circle.

In \cite{T06}, J. M. Thuswaldner studied the set ${\mathcal E}_{a_1,a_2}$.
In particular, he gave an explicit formula for the fractal dimension of the boundary of this set.

In this paper we propose to study the boundary of the Rauzy fractal set ${\mathcal
E}_{1,1,1}  \subset
  \mathbb{R} \times \mathbb{C}$.
We construct  a finite automaton characterizing its boundary.
As a consequence we prove that ${\mathcal
E}_{1,1,1}$ has $18$ neighborhoods where $6$ of them intersect the central tile ${\mathcal
E}_{1,1,1}$ in a point.
We also prove that the boundary can be generated by two subregions.
More precisely, the boundary of  ${\mathcal
E}_{1,1,1}$ is $ \bigcup_{i=1}^{18} X_i$ where $X_i,\; i=1, \ldots, 6$ are singletons, and for all $ i \in [7,18]$, there exist affine functions $f_{ij},\; j=1, \ldots, m_i$ and $g_{ij},\; j=1, \ldots, n_i$ from $\mathbb{R} \times \mathbb{C} $ to itself such that
$$X_i = \bigcup_{j=1}^{m_{i}}f_{ij} (X_{7}) \bigcup \bigcup_{j=1}^{n_{i}}g_{ij} (X_{8}). $$

\section{Notations, definitions and background}

\subsection{$\beta$-expansions}

Let $\beta >1$ be a real number. A $\beta$-representation of a
nonnegative real number $x $ is an infinite sequence $(x_{i})_{i
\leq k },\; x_{i} \in \mathbb{Z}^{+} = [0,+\infty [$, such that

$$
x=x_{k}\beta^{k}+
x_{k-1}\beta^{k-1}+\cdots +
x_{1}\beta+x_{0}+x_{-1}\beta^{-1}+x_{-2}\beta^{-2}+\cdots .
$$

where $k$ is an  integer.
It is denoted by

$$
x=x_{k}x_{k-1}\ldots x_{1}x_{0}.x_{-1}x_{-2}\ldots .
$$

A particular $\beta$-representation, called the $\beta$-expansion,
 is computed by
the "greedy algorithm" (see \cite{P60}):
 denote by $\lfloor y \rfloor$ and $\{y\}$ respectively the integer part
and the fractional part of a number $y$. There exists $k \in
\mathbb{Z}$ such that $\beta^{k} \leq x <\beta^{k+1}$. Let  $x_k=
\lfloor x/\beta ^{k}\rfloor$ and $r_k = \{ x/\beta ^{k}\}$. Then
for $i < k,$ put $x_i = \lfloor \beta r_{i+1} \rfloor$ and $r_i =
\{\beta r_{i+1}\}$. We get $$x= x_{k} \beta^{k}+ x_{k-1}
\beta^{k-1}+ \cdots$$ If $k <0 \; (x <1)$, we put $x_{0}=x_{-1}=
\cdots = x_{k+1}=0.$ If an expansion ends by infinitely many zeros,
it is said to be finite, and the ending zeros are omitted.

The  digits $x_{i}$ belong to the set  $A=\{0, \cdots, \beta -1\}$
if $\beta$ is an integer, or to the set $A=\{0, \cdots, \lfloor
\beta \rfloor\}$ if $\beta$ is not an integer. The $\beta$-expansion
of every positive real number $x$ is the lexicographically greatest among
all $\beta$-representations of $x$.

We denote by $\mbox{Fin}(\beta)$  the set of numbers which
have finite greedy $\beta$-expansion. Let $N \in \mathbb{Z}$, we
denote by $\mbox{Fin}_{N}(\beta)$  the set of numbers $x$ such that
in their $\beta$-expansion $(x_i)_{i\leq k}$, $x_{i}=0$ for all $i <N$. We will
sometimes denote a $\beta$-expansion $x_{n} \cdots x_{k}, \;n \geq
k$ by $(x_{i})_{k \leq i \leq n}.$
We put

$$
E_{\beta}=
\{ (x_{i})_{i \geq k } ; k \in \mathbb{Z}, \
 \forall n \geq k,\; (x_{i})_{k \leq i \leq n} \mbox { is a
 finite } \beta \mbox {-expansion} \}.$$
In the case where $\beta$ is the dominant root of the polynomial $P_{a_1, \ldots a_d}$,
it is known (see \cite{FS92}) that $E_{\beta}= \Dgeninf$.
We will need the two following classical lemmas.

\begin{lemma}
[\cite{P60}]
\label{lem-lex}
Let $x_{n}\cdots x_{0}$ and $y_{m}\cdots y_{0}$ be two
$\beta$-expansions.
Then, the following are equivalent
\begin{itemize}
\item
 $\sum_{i=0}^{n}x_{i}\beta^{i} <
\sum_{i=0}^{m}y_{i}\beta^{i} $,
\item
$x_{n}\cdots x_{0}
<_{lex} y_{m}\cdots y_{0}$,
\end{itemize}
where $<_{lex}$ is the lexicographical
order.
\end{lemma}

\begin{lemma}
[\cite{FS92}]
\label{fina}
If $\beta= \beta_1$, then $\mathbb{Z}[\beta]\cap [0, +\infty [ \subset  \hbox{\rm Fin}(\beta)$.
\end{lemma}

\subsection{Boundary of $\Egen$}

The coordinates of $\alpha$ have modulus strictly less than $1$.
Moreover, $0$ belongs to the interior of  $\Egen$ (\cite{A00}, see
also \cite{R82} for $\Er$). Hence, for all $z \in \mathbb{R}^{r-1}
\times \mathbb{C}^{s}$ there exists $k\in \NN$ such that
$\alpha^{k}z \in \Egen$.
Then, all $z\in \mathbb{R}^{r-1} \times
\mathbb{C}^{s}$ can be written as follows $ z= \sum_{i=l}
^{\infty}\varepsilon_{i}\alpha^{i} $, where $l \in \ZZ$ and
$(\varepsilon_{i})_{i \geq l} \in \Dgeninf$.
The sequence
$(\varepsilon_{i})_{i \geq l}$ is called $\alpha$-expansion of $z$.
We should remark that these $\alpha$-expansions are not unique :
some $z$ can have many different $\alpha$-expansions.
In \cite{M05}
it is proven that the points belonging to the boundary of $\Egen$
have at least two different $\alpha$-expansions.
These points are
characterized by the following theorem which is a straightforward
consequence of a result due to W. Thurston \cite{T90} (see also
\cite{M05}).

\begin{theo}
\label{auto}
There exists a finite automaton $B$ such that for all
distinct elements of $\Dgeninf$, $(b_{i})_{i \geq l}$ and $(c_{i})_{i \geq
l}$, the following are equivalent :

\begin{itemize}
\item
$\sum_{i=l}^{\infty}b_{i}\alpha^{i}
=\sum_{i=l}^{\infty}c_{i}\alpha^{i}$
\item
$((b_{i}, c_{i}))_{i\geq l}$ is recognizable by $B$ (i.e an infinite path in $B$ beginning in the initial state).
\end{itemize}
\end{theo}

The proof of this result does not give explicitly the states of the
automaton. In \cite{M98} is given an algorithm that gives these
states for ${\mathcal E}_{1,1}$.
In \cite{M06}, they
were given for ${\mathcal E}_{a_1,1}$ where $a_1 \geq 2$.

\section{Characterization of the boundary of $\Eart$}

In the sequel we suppose $d=3$ and $a_1=a_2=a_3=1$, and
$P(x)=P_{1,1,1} (x)= x^4-x^3-x^2-x-1=
(x-\beta_1)(x-\beta_2)(x-\beta_3)(x-\overline{\beta_3})$ where
$\beta_1,\beta_2,\beta_3$ are defined in Section \ref{intro}.
Approximations of these numbers are $\beta = \beta_1 = 1.9275 \dots ,\;
\beta_2 = -0.7748 \dots $ and $\beta_3 = -0.0763 \dots +i0.8147\dots $.
We recall that we defined for all $i\in \ZZ$, $\alpha^i=
(\beta_{2}^{i}, \beta_{3}^{i}).$

In this situation

\begin{align*}
\D={\mathcal D}_{1,1,1} &
=
\left\{
(\varepsilon_i)_{l \leq i \leq n} ; l,n\in \ZZ, \varepsilon_i\in \{ 0,1 \}, \varepsilon_i \varepsilon_{i-1}
\varepsilon_{i-2} \varepsilon_{i-3} \ne 1111,  l\leq i\leq n
\right\} ,\\
\D^\infty=
\D_{1,1,1}^\infty & = \left\{ (\varepsilon_i)_{i\geq l} ;
l\in \ZZ , (\varepsilon_{i})_{l \leq i \leq n} \in \D_{1,1,1},
 n\geq l \right\} \hbox{ and } \\
\E =
{\mathcal E}_{1,1,1} & =
\left\{ \sum_{i=4}^{+\infty}
\varepsilon_{i}\alpha^{i} ; (\varepsilon_i)_{i \geq 4} \in \D^\infty \right\}\\
&
= \left\{ \sum_{i=4}^{+\infty}
\varepsilon_{i}\alpha^{i} ; \varepsilon_i \in \{0, 1\},
\varepsilon_i \varepsilon_{i-1} \varepsilon_{i-2} \varepsilon_{i-3}
\ne 1111 , i \geq 4\right\} .
\end{align*}

An important and known result is:

\begin{prop}
\label{pavage}
The set ${\mathcal E}$ is compact,
connected and generates a periodic tiling of $\RR \times
\CC$ with group periods $ G= \mathbb{Z} \alpha^{0}+
\mathbb{Z} \alpha +  \mathbb{Z} \alpha^{2}$:

$$
\mathbb{R} \times \mathbb{C}=
\bigcup_{p \in G}({\mathcal E}+p) ,
$$

and the Lebesgue measure of $(\E+p) \cap (\E +q)$ is zero whenever $p\not =q$, $p,q \in G$.
\end{prop}

\begin{proof}
The proof can easily be deduced from \cite{R82} or \cite{AI01}.
\end{proof}

\subsection{Definition of the automaton recognizing the points with at least two expansions}
\label{def-auto}

In the sequel we proceed to the construction of the automaton $\A$ that
 characterize the boundary of $\E$.
This characterization will be proven in Section \ref{subsec-characterization}.

The set of states of the automaton $\A$ is

$$
\begin{array}{lll}
S
= & &
\left\{ \pm \sum_{i=0}^{3}c_{i}\alpha^{i} ; c_{0}c_{1}c_{2}c_{3}
\ne 1111 , c_i \in \{ 0,1 \} ,  0\leq i \leq 3 \right\}  \\
& \bigcup &
\left\{
\pm (\alpha^{-1}+1 + \alpha^2),
\pm (\alpha^{-2}+\alpha^{-1} + \alpha),
\pm (\alpha^{-3}+\alpha^{-2} +1+ \alpha^3)
\right\} .
\end{array}
$$

Let $s$ and $t$ be two states. The set of edges is the set of
$(s,(a,b),t)\in S \times \{ 0,1 \}^2 \times S$ satisfying $t = \frac
{s}{\alpha} + (a-b) \alpha^{3}$. The set of initial states is $\{
0 \}$ and the set of states is $S$.
A path (resp. infinite path) of $\A$ is a sequence $(a_n,b_n)_{k\leq n\leq l}$ (resp. $(a_n,b_n)_{n\geq k}$) such that there exists a sequence
$(e_n)_{k\leq n\leq l+1}$ (resp. $(e_n)_{n\geq k}$) of elements of $S$ for which $(e_n , (a_n , b_n) , e_{n+1})$ belongs to $S$ for all $n\in \{ k,k+1, \dots , l+1\}$
(resp. $n\geq k$). We say it starts in the initial state when $e_k = 0$.
The automaton is explicitly defined in the Annexe at the end of this paper.

Let us explain the behavior of this automaton.
Let $\varepsilon = (\varepsilon_{i})_{i\geq l}$ and $\varepsilon' = (\varepsilon'_{i})_{i\geq l}$ belonging to $\Dinfty$,
$x=\sum_{i=l}^{\infty}\varepsilon_{i}\alpha^{i}$ and
$y=\sum_{i=l}^{\infty}\varepsilon'_{i}\alpha^{i}$.
For all
$k\geq l$ we set

\begin{align}
\label{def-Ak}
A_k (\varepsilon ,\varepsilon')
=
\alpha^{-k+3}\sum_{i=l}^{k}(\varepsilon_{i}-\varepsilon'_{i})\alpha^{i}
\end{align}

In Subsection \ref{subsec-characterization}
we will prove that $x=y$ if and only if all the $A_k$, $k\geq l$, belong to $S$.
But as, for
all $k\geq l$, we have

\begin{eqnarray}
\label{rel-Ak}
A_{k+1} (\varepsilon ,\varepsilon') = \frac {A_{k}(\varepsilon ,\varepsilon')}{\alpha} + (\varepsilon_{k+1}- \varepsilon'_{k+1})
\alpha^{3} ,
\end{eqnarray}

this means that $x=y$ if and only if

$$
\left(0, \left(\varepsilon_l , \varepsilon'_l \right) ,A_l \left(\varepsilon ,\varepsilon'\right)\right)
\left( (A_k (\varepsilon ,\varepsilon'), (\varepsilon_{k+1} , \varepsilon'_{k+1} ),A_{k+1} (\varepsilon ,\varepsilon')) \right)_{k\geq l}
$$

is an infinite sequence of edges of $S$ starting in the initial state.
And, this is equivalent to say that $(\varepsilon_i , \varepsilon'_i)_{i\geq l}$ is an infinite path of $\A$ starting in the initial state.

Let us give an example on how we can use this automaton to obtain information
about the digits of $x$ and $y$.
Let $s$ be the smallest integer such that $\varepsilon_{s} \ne \varepsilon'_{s}$.
Hence $A_{i} (\varepsilon ,\varepsilon') = 0$
for $i \in \{l, \cdots, s-1\}$.
Suppose  $\varepsilon_{s} >\varepsilon'_{s}$, that is $\varepsilon_s = 1 $ and $\varepsilon'_s = 0$.
Then,
$A_{s}= \alpha^{3}$.
From \eqref{rel-Ak} we deduce $ A_{s+1}(\varepsilon ,\varepsilon') =
\alpha^{2} + (\varepsilon_{s+1}- \varepsilon'_{s+1}) \alpha^{3}$ which should belong to $S$.
Hence
$A_{s+1} (\varepsilon ,\varepsilon')=\alpha^{2}\in S$ if $\varepsilon_{s+1}= \varepsilon'_{s+1}$, and,
$A_{s+1} (\varepsilon ,\varepsilon')=\alpha^{2}+ \alpha^{3}\in S$ if $(\varepsilon_{s+1}, \varepsilon'_{s+1})= (1,0)$.
Hence, $(\alpha^3 , (1,0) , \alpha^2+\alpha^3 )$, $(\alpha^3 , (0,0) , \alpha^2 )$ and $(\alpha^3 , (1,1) , \alpha^2 )$
are edges coming from the state $\alpha^3$.
Let us explain why $(\alpha^3 , (0,1) , \alpha^2-\alpha^3 )$ is not an edge, and hence why we cannot have
$(\varepsilon_{s+1},\varepsilon'_{s+1}) = (0,1)$.
We should have that $\alpha^2-\alpha^3 = -\alpha^{-1} - 1-\alpha$ belongs to $S$.
Then $\beta$ should satisfies the same equality.
Hence $\beta^{-1} +1+\beta$ should belong to

\begin{align*}
&\left\{ \sum_{i=0}^{3}c_{i}\beta^{i} ; c_{0}c_{1}c_{2}c_{3}
\ne 1111 , c_i \in \{ 0,1 \} ,  0\leq i \leq 3 \right\}  \\
 \bigcup &
\left\{
(\beta^{-1}+1 + \beta^2),
(\beta^{-2}+\beta^{-1} + \beta),
(\beta^{-3}+\beta^{-2} +1+ \beta^3)
\right\} ,
\end{align*}

which is not possible by Lemma \ref{lem-lex}.

\subsection{Characterization of the points with at least two expansions}
\label{subsec-characterization}

\begin{lemma}
\label{maple}
Let $(\varepsilon_{i})_{i \geq 0},(\varepsilon_{i}')_{i \geq 0} \in \Dinfty$.
Then,
\begin{align*}
\left|
\sum_{i=0}^{+ \infty}(\varepsilon_{i}-\varepsilon_{i}') \beta_2^{i}
\right|
\leq \frac{1}{1+\beta_2}
, \ \
\left|
\sum_{i=0}^{+ \infty}(\varepsilon_{i}-\varepsilon_{i}')\beta_3^{i}
\right|
 \leq  \frac{C}{1 - \vert \beta_3 \vert^{6}} .
\end{align*}
where $C= \max
\left
\{
\left| \sum_{i=0}^{5} (c_i-d_i) \beta_3^{i}
\right|
; (c_i)_{0\leq i\leq 5} \in \D , ( d_i)_{0\leq i\leq 5} \in \D
\right\}$.
\end{lemma}

\begin{proof}
The second inequality is easy to establish.
For the first inequality, as $-1 <\beta_2 <0$, all sequences $(c_i)_{i \geq 0}$ which terms
are $0$ or $1$  satisfy the following inequality~:

$$
\frac{\beta_2}{1-\beta_2^2}=\sum_{i=0}^{+\infty} \beta_2 ^{2i+1}
\leq \sum_{i=0}^{+\infty} c_i \beta_{2}^{i}   \leq
\sum_{i=0}^{+\infty} \beta_2^{2i} = \frac{1}{1-\beta_2^2}.
$$

This achieves the proof.
\end{proof}

For all $\varepsilon = (\varepsilon_{i})_{i \geq l}$ and
$\varepsilon'=(\varepsilon'_{i})_{i \geq l}$ belonging to $\D^\infty$, we set

$$
S(\varepsilon, \varepsilon')=\{ A_k (\varepsilon,\varepsilon') ; k\geq l \}
=
\left\{
\alpha^{-k+3}\sum_{i=l}^{k}(\varepsilon_{i}-\varepsilon'_{i})\alpha^{i} ;
k \geq l
\right\} .
$$

\begin{prop}
\label{egalite}
Let $x=\sum_{i=l}^{\infty}\varepsilon_{i}\alpha^{i}$,
$y=\sum_{i=l}^{\infty}\varepsilon'_{i}\alpha^{i}$, where
 $\varepsilon = (\varepsilon_{i})_{i \geq l}$ and
$\varepsilon' =(\varepsilon'_{i})_{i \geq l}$ belong to $\D^\infty$.
Then, $x=y$ if
and only if the set $S(\varepsilon, \varepsilon')$ is finite.
Moreover

\begin{align*}
S(\varepsilon ,\varepsilon')
\subset
S
= &
\left\{ \pm \sum_{i=0}^{3}c_{i}\alpha^{i} ; (c_{i})_{0\leq i\leq 3}
\in \D \right\}  \\
& \bigcup
\left\{
\pm (\alpha^{-1}+1 + \alpha^2),
\pm (\alpha^{-2}+\alpha^{-1} + \alpha),
\pm (\alpha^{-3}+\alpha^{-2} +1+ \alpha^3)
\right\}
\end{align*}

and

$$
S=\bigcup_{(\varepsilon ,\varepsilon')\in \Delta  } S (\varepsilon ,\varepsilon') ,
$$

where $\Delta = \left\{ \left((\varepsilon_i)_{i\geq l}, (\varepsilon'_i)_{i\geq l}\right) \in \Dinfty\times \Dinfty ; \sum_{i=l}^{\infty}\varepsilon_{i}\alpha^{i}=\sum_{i=l}^{\infty}\varepsilon'_{i}\alpha^{i}  \right\}$.
\end{prop}

\begin{proof}
It is easy to establish that if $S(\varepsilon , \varepsilon')$ is finite then $x=y$.
Let us prove the reciprocal.
Let $x=\sum_{i=l}^{\infty}\varepsilon_{i}\alpha^{i}=\sum_{i=l}^{\infty}\varepsilon_{i}'\alpha^{i}=y$ with $\varepsilon = (\varepsilon_{i})_{i \geq l}$ and
$\varepsilon'=(\varepsilon'_{i})_{i \geq l}$ belonging to $\D^\infty$.
Let us prove that $A_k = A_k (\varepsilon , \varepsilon' )$ belongs to $S$ for all $k\geq l$.
As $x=y$, for all $k\geq l$, we have

\begin{align}
 \label{Ak}
A_k
=
 \sum_{i=k+1}^{\infty}(\varepsilon'_{i}-\varepsilon_{i})\alpha^{i-k+3}
 =
 \sum_{i=4}^{\infty}(\varepsilon'_{i+k-3}-\varepsilon_{i+k-3})\alpha^{i} .
\end{align}

Let us fix $k\geq l$ and assume $A_k \not = 0$.
From \eqref{def-Ak}, we deduce there exist $n, p,q, r \in \mathbb{Z}$ such that

\begin{align}
\label{Ak-alg}
A_k= n \alpha^{3}+ p \alpha^2+ q \alpha + r .
\end{align}

But $n \beta^{3}+ p \beta^{2}+ q \beta + r $ or $-(n \beta^{3}+ p
\beta^{2}+ q \beta + r) $  belongs to $\mathbb{Z}[\beta] \cap
\mathbb{R}^{+}$, which is contained in ${\rm Fin}(\beta)$
(see Lemma \ref{fina}).
We deduce there exists $(c_{i})_{s \leq i \leq m} \in
{ \mathcal D }$ such that $c_m=1$ and

\begin{align}
 \label{Ak-alg2}
n \beta^{3}+ p \beta^{2}+ q \beta + r = \pm\sum_{i=s}^{m}
c_{i}\beta^{i} .
\end{align}

We suppose it is equal to $\sum_{i=s}^{m}c_{i}\beta^{i}$.
The other case can be treated in the same way.
As $\beta $, $\beta_2$ and $\beta_3$ are algebraically conjugate, from \eqref{def-Ak}, \eqref{Ak-alg} and \eqref{Ak-alg2} we have

\begin{align}
\label{eq-ali}
\beta^{-k+3}\sum_{i=l}^{k} \varepsilon_{i} \beta^{i} & =
\beta^{-k+3}\sum_{i=l}^{k} \varepsilon'_{i}\beta^{i} +
\sum_{i=s}^{m}c_{i}\beta^{i}.
\end{align}

From Lemma \ref{lem-lex},
$\beta^{-k+3}\sum_{i=l}^{k}\varepsilon_{i} \beta^{i} < \beta^4$,
consequently $m \leq 3$.
Setting $c_i = 0$ for $i>m$, we have

\begin{align}
\label{Akc}
A_k  =  \sum_{i=s}^{3} c_{i}\alpha^{i} .
\end{align}

Remark that if $s\geq 0$ then $A_k$ belongs to $S$.
Hence we suppose $s\leq -1$.

Suppose $s=-1$ and $c_{-1} = 1$.
Let us show that $A_k $ is equal to $ \alpha^{-1}+1 + \alpha^2 $ and consequently belongs to $S$.
In order to do so, we show that the other cases are not possible.
Using Lemma \ref{maple} and \eqref{Ak}, the first entry of $(A_k)$, $(A_k)_1$,
should satisfy $|(A_k)_1| \leq \beta_2^4(1+\beta_2)^{-1}$
which is less than $a=1.6004$.
This excludes the following points~: $\alpha^{-1} + \alpha + \alpha^2 + \alpha^3 $, $ \alpha^{-1} + \alpha + \alpha^3 $, $\alpha^{-1} + \alpha$ and $\alpha^{-1} + \alpha^3$ because the absolute value of their first entries is greater than the value below it in the following array :

$$
\begin{array}{|c|c|c|c|}
\hline
\beta_2^{-1} + \beta_2 + \beta_2^2 + \beta_2^3  & \beta_2^{-1} + \beta_2 + \beta_2^3 & \beta_2^{-1} + \beta_2 & \beta_2^{-1} + \beta_2^3 \\
\hline
1.9 & 2.5 & 2.0 & 1.7\\
\hline
\end{array}
$$

In the same way we should have $|(A_k)_2| \leq C|\beta_3|^4(1-|\beta_3|^6)^{-1}$
which is less than $b=1,8120$.
This excludes the following points : $\alpha^{-1} + 1+ \alpha^3 $, $ \alpha^{-1} + \alpha^2 + \alpha^3 $ and $ \alpha^{-1} + 1 + \alpha^2 + \alpha^3 $,
because the absolute value of their second entries is greater than the value below it in the following array :

$$
\begin{array}{|c|c|c|}
\hline
\beta_3^{-1} + 1+  \beta_3^3  & \beta_3^{-1} + \beta_3^2 + \beta_3^3& \beta_3^{-1} + 1+\beta_3^2 + \beta_3^3 \\
\hline
2.0 & 1.9 & 1.9 \\
\hline
\end{array}
$$

In order to exclude the other cases, except $\frac{1}{\alpha} +1+\alpha^2$, we used \eqref{rel-Ak} to compute
$A_{k+i}$, $i\geq 1$.
Let us explain the strategy.
Suppose neither $(A_k)_1$ nor $(A_k)_2$ is greater than respectively $a$ and $b$.
Then, we compute $A_{k+1}$ using  \eqref{rel-Ak}.
We have three possible values : $\frac{A_k}{\alpha}$, $\frac{A_k}{\alpha} +\alpha^3$ and $\frac{A_k}{\alpha}-\alpha^3$.
To check that $A_k$ does not belong to $S$, it suffices to show that for all these values,
either the first entry or the second is respectively greater than $a$ or $b$.
If it is not the case, for each value that does not satisfy this (both entries are less than, respectively, $a$ and $b$)
we apply again this strategy.
Applying this just once we show that $\frac{1}{\alpha} +1+ \alpha +\alpha^3$ does not belong to $S$.
The values of the relevant entries are in the following array and should be read in the following way :
The value ($1.9$ for example) below a relevant entry of $A_{k+1}$ (resp. $\frac{1}{\beta_2^{2}} +\frac{1}{\beta_2}+ 1 +\beta_2^2 $)
is greater than the absolute value of the relevant entry : $|\frac{1}{\beta_2^{2}} +\frac{1}{\beta_2}+ 1 +\beta_2^2 |>1.9$.

$$
\begin{array}{|c|c|c|c|}
\hline
A_k & \frac{1}{\alpha} +1+ \alpha +\alpha^3  & & \\
\hline
A_{k+1} & \frac{1}{\beta_2^{2}} +\frac{1}{\beta_2}+ 1 +\beta_2^2
        & \frac{1}{\beta_3^{2}} +\frac{1}{\beta_3}+ 1 +\beta_3^2 +\beta_3^3
        & \frac{1}{\beta_2^{2}} +\frac{1}{\beta_2}+ 1 +\beta_2^2 -\beta_2^3 \\
\hline
&1.9 & 1.9 & 2.4 \\
\hline
\end{array}
$$

For the following case, $\frac{1}{\alpha} +1$, we need to apply the strategy twice because
for $A_{k+1} = \frac{1}{\beta_3^{2}} +\frac{1}{\beta_3} -\beta_3^3 $ both entries are respectively less than
$a$ and $b$.

$$
\begin{array}{|l|l|l|l|}
\hline
A_k & \frac{1}{\alpha} +1  &   &\\
\hline
A_{k+1} & \frac{1}{\beta_3^{2}} +\frac{1}{\beta_3}
        & \frac{1}{\beta_3^{2}} +\frac{1}{\beta_3} +\beta_3^3
        & \\
\hline
& 1.83 & 2.0  & \\
 \hline
 A_{k+2} & \frac{1}{\beta_3^{3}} +\frac{1}{\beta_3^{2}}-\beta_3^2
         & \frac{1}{\beta_3^{3}} +\frac{1}{\beta_3^{2}}-\beta_3^2 +\beta_3^3
         & \frac{1}{\beta_3^{3}} +\frac{1}{\beta_3^{2}}-\beta_3^2 -\beta_3^3\\
\hline
 & 2.1 & 1.63 & 2.7 \\
\hline
\end{array}
$$

For the case $A_k = \frac{1}{\alpha} +1 +\alpha $ we need two steps because at the first one both
$\frac{1}{\beta_2^{2}} +\frac{1}{\beta_2}+1 $ and $\frac{1}{\beta_2^{2}} +\frac{1}{\beta_2}+1 +\beta_3^3$
have entries less than, respectively, $a$ and $b$.

$$
\begin{array}{|l|l|l|l|}
\hline
A_k & \frac{1}{\alpha} +1 +\alpha &   &\\
\hline
A_{k+1} & & & \frac{1}{\beta_2^{2}} +\frac{1}{\beta_2}+1 -\beta_2^3  \\
\hline
& & & 1.84 \\
 \hline
 A_{k+2} & \frac{1}{\beta_2^{3}} +\frac{1}{\beta_2^{2}}+\frac{1}{\beta_2}
         & \frac{1}{\beta_2^{3}} +\frac{1}{\beta_2^{2}}+\frac{1}{\beta_2}+ \beta_2^3
         & \frac{1}{\beta_3^{3}} +\frac{1}{\beta_3^{2}}+\frac{1}{\beta_3}- \beta_3^3 \\
\hline
 & 1.77& 2.23 & 1.818  \\
\hline
 & \frac{1}{\beta_3^{3}} +\frac{1}{\beta_3^{2}}+\frac{1}{\beta_3}+\beta_3^2
 & \frac{1}{\beta_2^{3}} +\frac{1}{\beta_2^{2}}+\frac{1}{\beta_2}+\beta_2^2+\beta_2^3
 & \frac{1}{\beta_3^{3}} +\frac{1}{\beta_3^{2}}+\frac{1}{\beta_3}+\beta_3^2- \beta_3^3  \\
\hline
& 1.86& 1.63 & 2.24  \\
\hline
\end{array}
$$

For the three following cases, $\frac{1}{\alpha} +\alpha^2$, $\frac{1}{\alpha}$ and $\frac{1}{\alpha} +\alpha + \alpha^2$,  we need three steps.

$$
\begin{array}{|l|l|l|l|}
\hline
A_k & \frac{1}{\alpha} +\alpha^2 &   &\\
\hline
A_{k+1} & \frac{1}{\beta_3^{2}} +\beta_3  &
        & \frac{1}{\beta_3^{2}} +\beta_3 - \beta_3^3  \\
\hline
& 1.89 & & 2.34   \\
 \hline
 A_{k+2} & \frac{1}{\beta_3^{3}} +1 + \beta_3^2 &
         & \frac{1}{\beta_3^{3}} +1 + \beta_3^2 - \beta_3^3    \\
\hline
 & 1.83 & & 2.26   \\
\hline
A_{k+3} & \frac{1}{\beta_3^{4}} +\frac{1}{\beta_3} + \beta_3 + \beta_3^2
        & \frac{1}{\beta_3^{4}} +\frac{1}{\beta_3} + \beta_3 + \beta_3^2+\beta_3^3
        & \frac{1}{\beta_2^{4}} +\frac{1}{\beta_2} + \beta_2 + \beta_2^2-\beta_2^3 \\
\hline
& 1.818& 2.32 & 1.77  \\
\hline
\end{array}
$$

$$
\begin{array}{|l|l|l|l|}
\hline
A_k & \frac{1}{\alpha} &   &\\
\hline
A_{k+1} & \frac{1}{\beta_2^{2}}   &
        & \frac{1}{\beta_2^{2}} - \beta_2^3  \\
\hline
& 1.66 & & 2.13   \\
 \hline
 A_{k+2} &
         & \frac{1}{\beta_2^{3}} + \beta_2^2+\beta_2^3
         & \frac{1}{\beta_3^{3}} + \beta_3^2-\beta_3^3    \\
\hline
& & 2.01 & 2.17   \\
\hline
A_{k+3} & \frac{1}{\beta_3^{4}} + \beta_3
        & \frac{1}{\beta_3^{4}} + \beta_3 +\beta_3^3
        & \frac{1}{\beta_3^{4}} + \beta_3-\beta_3^3 \\
\hline
& 2.00 & 2.21 & 1.92  \\
\hline
\end{array}
$$

$$
\begin{array}{|l|l|l|l|}
\hline
A_k & \frac{1}{\alpha} +\alpha + \alpha^2 &   &\\
\hline
A_{k+1} & \frac{1}{\beta_2^{2}} + 1 + \beta_2 &
        & \frac{1}{\beta_2^{2}} + 1 + \beta_2 -\beta_2^3  \\
\hline
& 1.89 & & 2.35   \\
 \hline
 A_{k+2} & \frac{1}{\beta_2^{3}} + \frac{1}{\beta_2} + 1 +\beta_2^2
         & \frac{1}{\beta_2^{3}} + \frac{1}{\beta_2} + 1 +\beta_2^2+\beta_2^3  &  \\
\hline
 & 1.84 & 2.30 &  \\
\hline
A_{k+3} & \frac{1}{\beta_2^{4}} +\frac{1}{\beta_2^2} + \frac{1}{\beta_2} + \beta_2
        & \frac{1}{\beta_3^{4}} +\frac{1}{\beta_3^{2}} + \frac{1}{\beta_3} + \beta_3
        & \frac{1}{\beta_2^{4}} +\frac{1}{\beta_2^{2}} + \frac{1}{\beta_2} + \beta_2  \\
        &  -\beta_2^2
        & -\beta_3^2 +\beta_3^3
        & -\beta_3^2 -\beta_3^3 \\
\hline
& 1.77& 1.818 & 2.23   \\
\hline
\end{array}
$$

Hence the only possible $A_k$ (with $c_{-1} = 1$) is $\frac{1}{\alpha} +1+\alpha^2$.

Suppose now $s\leq -2$ and $c_{s} = 1$.
It is useful for the sequel to remark that $U=(u_i)_{i\geq s}=(c_s , c_{s+1} , \dots , c_2,c_3,\varepsilon_{k+1} , \varepsilon_{k+2},\varepsilon_{k+3}, \dots )$
belongs to $\Dinfty$.
Indeed, if $c_3 = 0$, it is clear.
If $c_3 = 1$ and $c_2=0$ then by \eqref{eq-ali}, $\varepsilon_k = 1$.
Hence $\varepsilon_{k+1}\varepsilon_{k+2}\varepsilon_{k+3}\not = 111$ and $U$ belongs to $\D^\infty$.
The other cases can be treated in the same way.

Using \eqref{Ak} and \eqref{Akc} we obtain

$$
u=\sum_{i=s}^{3} c_{i}\alpha^{i} + \sum_{i=4}^{\infty}\varepsilon_{i+k-3}\alpha^{i}
=
\sum_{i=4}^{\infty}\varepsilon'_{i+k-3}\alpha^{i} =v.
$$

We set $V = (\varepsilon'_{i+k-3})_{i\geq s}$ where
$\varepsilon'_{i+k-3} = 0$ when $s\leq i\leq 3$.
Then $ (U,V)$ belongs to $\Delta$, $A_{s+4} (U,V) = c_s\alpha^{-1} + c_{s+1} +c_{s+2}\alpha +c_{s+3}\alpha^2 +
c_{s+4}\alpha^3$ and $A_3 (U,V) = \sum_{i=s}^{3} c_{i}\alpha^{i}$.
Doing what we did for $x$ and $y$ to $u$ and $v$ we obtain that $A_{s+4} (U,V) = \alpha^{-1} +1+\alpha^2$.
Let us show that for all $n\geq s+5$, $A_{n} (U,V)$ belongs to

$$
\C=
\left\{
\pm (\alpha^{-2}+\alpha^{-1} + \alpha),
\pm (\alpha^{-3}+\alpha^{-2} +1+ \alpha^3)
\right\} .
$$

This will imply that $A_k$ belongs to $S$ for all $k\geq l$.
We have that $A_{s+5}(U,V)$ belongs to

$$
\left\{ \frac{1}{\alpha^2} + \frac{1}{\alpha} +\alpha ,\frac{1}{\alpha^2} + \frac{1}{\alpha} +\alpha +\alpha^3 , \frac{1}{\alpha^2} + \frac{1}{\alpha} +\alpha -\alpha^3 \right\} .
$$

The third one can be excluded because $\frac{1}{\beta_3^2} + \frac{1}{\beta_3} +\beta_3 -\beta_3^3 \geq 1.85$.
We proceed as before to exclude the second element :

$$
\begin{array}{|l|l|l|l|}
\hline
A_k & \frac{1}{\alpha^2} +\frac{1}{\alpha } + \alpha + \alpha^3 &   &\\
\hline
A_{k+1} & \frac{1}{\beta_3^{3}} + \frac{1}{\beta_3^{2}} +1 + \beta_3^2 &
        & \frac{1}{\beta_3^{3}} + \frac{1}{\beta_3^{2}} +1 + \beta_3^2 - \beta_3^3 \\
\hline
& 2.00 & & 2.55   \\
 \hline
 A_{k+2} & \frac{1}{\beta_3^{4}} + \frac{1}{\beta_3^{3}} +\frac{1}{\beta_3} + \beta_3
         & \frac{1}{\beta_3^{4}} + \frac{1}{\beta_3^{3}} +\frac{1}{\beta_3} + \beta_3 + \beta_3^3
         & \frac{1}{\beta_3^{4}} + \frac{1}{\beta_3^{3}} +\frac{1}{\beta_3} + \beta_3 - \beta_3^3   \\
\hline
         & 2.45 & 2.54 & 2.47  \\
\hline
\end{array}
$$

Consequently, $A_{s+5} (U,V) = \frac{1}{\alpha^2} + \frac{1}{\alpha} +\alpha$.
We deduce $A_{s+6} (U,V) = \frac{1}{\alpha^3} + \frac{1}{\alpha^2} +1 + \alpha^3$ because
$\frac{1}{\beta_3^3} + \frac{1}{\beta_3^2} +1 > 2.03$ and $\frac{1}{\beta_3^3} + \frac{1}{\beta_3^2} +1 - \beta_3^3 > 2.56$.
Once again, $A_{s+7} (U,V) = \frac{1}{\alpha^4} + \frac{1}{\alpha^3} + \frac{1}{\alpha}+ \alpha^2 -\alpha^3$ because
$\frac{1}{\beta_3^4} + \frac{1}{\beta_3^3} +\frac{1}{\beta_3}+ \beta_3^2 > 1.85$ and $\frac{1}{\beta_3^4} + \frac{1}{\beta_3^3} +\frac{1}{\beta_3}+ \beta_3^2 +\beta_3^3 > 2.16$.
But an easy computation leads to $\frac{1}{\alpha^4} + \frac{1}{\alpha^3} + \frac{1}{\alpha}+ \alpha^2 -\alpha^3 = -(\frac{1}{\alpha^2}+\frac{1}{\alpha} +\alpha ) = -A_{s+5} (U,V)$.
Then continuing in the same way we can check $A_n (U,V) = -A_{n+2} (U,V)$ and $A_n\in \C$ for $n \geq s+5$.
As $3\geq s+5$, we obtain that $A_k (\varepsilon , \varepsilon' ) = A_3(U,V)$ belongs to $\C$.
Thus $S(\varepsilon , \varepsilon' ) $ is included in $S$.

To complete the proof we should show that each element of $S$ belongs to $\Gamma = \cup_{(\varepsilon , \varepsilon' )\in \Delta} S(\varepsilon, \varepsilon' )$.

Remark that if $A_k$ belongs to $\Gamma$ then $-A_k$ also belongs to $\Gamma$.
Consequently it is sufficient to consider the cases where
$A_k = \sum_{i=0}^3 c_i \alpha^i$ with $(c_i)_{0\leq i\leq 3} \in \D$ or $A_k =\alpha^{-1}+1 + \alpha^2$,
$\alpha^{-2}+\alpha^{-1} + \alpha$ or
$\alpha^{-3}+\alpha^{-2} +1+ \alpha^3$.

Notice that we have

\begin{align*}
-\alpha^3 & = \sum_{i=1}^{+\infty} (\alpha^{4i} + \alpha^{4i+1}+ \alpha^{4i+2})= 1+
\alpha+ \alpha^2+ \sum_{i=1}^{+\infty} (\alpha^{4i+1} +
\alpha^{4i+2}+ \alpha^{4i+3})\\
 &=\alpha+ \alpha^2+ \alpha^4+ \sum_{i=1}^{+\infty}
(\alpha^{4i+2} + \alpha^{4i+3}+ \alpha^{4i+4}) \\
& =
\alpha^2+\alpha^4+\alpha^5+ \sum_{i=1}^{+\infty} (\alpha^{4i+3} +
\alpha^{4i+4}+ \alpha^{4i+5}).
\end{align*}

Hence, $1+\alpha +\alpha^2$, $\alpha +\alpha^2$ and $\alpha^2$ belong to $\Gamma$.
Multiplying by $\alpha$ we deduce $\alpha +\alpha^2 + \alpha^3$, $\alpha^2 +\alpha^3$ and $\alpha^3$ belong to $\Gamma$.
Now subtracting $-\alpha^2$ we obtain $1+\alpha$ and $\alpha$ belong to $\Gamma$.
We have that $1$ belongs to $\Gamma$
because $\sum_{i=1}^{+\infty} \alpha^{4i} = 1 +\sum_{i=1}^{+\infty}
\alpha^{4i+1}$.
Now, $1+\alpha^2$ belongs to $\Gamma$ because

\begin{align*}
\sum_{i=2}^{\infty} \alpha^{2i}= 1+ \alpha^2+
 \sum_{i=2}^{\infty} \alpha^{2i+1} .
 \end{align*}

Multiplying by $\alpha$ we deduce $\alpha + \alpha^3$ belongs to $\Gamma$.
Because

\begin{align}
\label{eq-5}
\alpha^{-3}+ \alpha^{-2}+ 1+\alpha^{3} +\sum_{i=1}^{\infty} (\alpha^{4i+2}+\alpha^{4i+3})  = \sum_{i=1}^{\infty} (\alpha^{4i}+\alpha^{4i+1}),
\end{align}

we obtain that $\alpha^{-3}+ \alpha^{-2}+ 1+\alpha^{3}$ belongs to $\Gamma$.
Multiplying \eqref{eq-5} by, respectively, $\alpha$ and $\alpha^2$ we obtain, respectively, that
$\alpha^{-2}+\alpha^{-1} +\alpha$ and $\alpha^{-1}+1 +\alpha^2$ belong to $\Gamma$.
From
\begin{align*}
\alpha^{4}+\sum_{i=1}^{\infty}
\alpha^{4i+3}  = 1+\alpha^3 +\sum_{i=1}^{\infty} \alpha^{4i+1}
= 1+\alpha + \alpha^3 + \sum_{i=1}^{\infty} \alpha^{4i+2}
\end{align*}

it is clear $1+\alpha^3$ and $1+\alpha+\alpha^3$ belong to $\Gamma$.

The equality
\begin{align*}
\alpha^{4}+\sum_{i=1}^{\infty}
\alpha^{4i+2}  = 1+\alpha^2+ \alpha^3 + \sum_{i=1}^{\infty} \alpha^{4i+1}
\end{align*}
implies that $1+\alpha^2+\alpha^3$ belongs to $\Gamma$ and achieves the proof.
\end{proof}

\begin{coro}
\label{coro-auto}
Let $\A$ be the automaton defined in Subsection \ref{def-auto}.
Then,
for all $(\varepsilon_i)_{i\geq l}$ and $(\varepsilon'_i)_{i\geq l}$
belonging to $\Dinfty$ the following assertions are equivalent :

\begin{itemize}
\item
$\sum_{i\geq l} \varepsilon_i \alpha^i = \sum_{i\geq l} \varepsilon'_i \alpha^i$ ;
\item
$(\varepsilon_i , \varepsilon'_i)_{i\geq l}$ is an infinite path in $\A$ beginning in the initial
state.
\end{itemize}
\end{coro}

\subsection{Neighbors of $\E$}

Here we prove that the set $\E$ has $18$ neighbors where $6$ of them have an intersection with $\E$ reduced to a singleton, and that the boundary can be generated by just $2$ subregions.

\begin{lemma}
\label{lemme-tec1}
Let $(\varepsilon_i)_{i \geq 4}$ and $(\varepsilon'_i)_{i \geq l}$ be two elements of $\Dinfty$
such that $\sum_{i=4}^{\infty}\varepsilon_{i}\alpha^{i} =\sum_{i=l}^{\infty}\varepsilon'_{i}\alpha^{i}$,
where $ l < 4$ and $\varepsilon'_{l}=1$, then $\varepsilon'_{l}\alpha^{l}+ \varepsilon'_{l+1}\alpha^{l+1}
\cdots + \varepsilon'_{3}\alpha^{3}$ belongs to $ S$.
In particular $l \geq -3$ and

\begin{align*}
\varepsilon'_{l}\alpha^{l}+
\cdots + \varepsilon'_{3}\alpha^{3} & = \alpha^{-3}+ \alpha^{-2}+ 1
 + \alpha^{3} & \mbox { if } l=-3,\\
\varepsilon'_{l}\alpha^{l}+
\cdots + \varepsilon'_{3}\alpha^{3}& = \alpha^{-2}+ \alpha^{-1}+ \alpha
 &\mbox { if } l=-2,\\
\varepsilon'_{l}\alpha^{l}+
\cdots + \varepsilon'_{3}\alpha^{3}& = \alpha^{-1}+1 + \alpha^{2}
 &\mbox { if } l=-1.
 \end{align*}
\end{lemma}
\begin{proof}
Let $(\varepsilon_i)_{i \geq 4}$ and $(\varepsilon'_i)_{i \geq l}$ be two elements of $\Dinfty$
such that $\sum_{i=4}^{\infty}\varepsilon_{i}\alpha^{i} =\sum_{i=l}^{\infty}\varepsilon'_{i}\alpha^{i}$
where $ l < 4$ and $\varepsilon'_{l}=1$.
From Proposition \ref{egalite}, for all $l \leq i \leq 3,$
$\varepsilon'_{l}\alpha^{i}+\varepsilon'_{l+1}\alpha^{i+1}
\cdots + \varepsilon'_{l-i+3}\alpha^{3}$ belongs to $S$.
In particular, for $i=l$, we obtain the result.
\end{proof}

\begin{lemma}
\label{lemme-tec2}
Let $u\in S$.
Then, there exists $(\varepsilon_i)_{i \geq 4}$ and $(\varepsilon'_i)_{i \geq 4}$ belonging to $\Dinfty$
such that $\sum_{i=4}^{\infty}\varepsilon_{i}\alpha^{i} =u + \sum_{i=4}^{\infty}\varepsilon'_{i}\alpha^{i}$.
\end{lemma}

\begin{proof}
This comes from Proposition \ref{egalite} and the identity \eqref{Ak}.
\end{proof}

In our context, Lemma 2 in \cite{M05} can be formulated in the following way.

\begin{lemma}
\label{lemme-ali}
Let $x \in \RR \times \CC$, then $x$ belongs to the boundary of $\E$
if and only if
there exists $l \leq 3$ such that $x= \sum_{i=4}^{+\infty} \epsilon_i \alpha^{i}=
 \sum_{i=l}^{+\infty} \epsilon'_i \alpha^{i}$, where $( \epsilon_i)_{i \geq 4}$ and
$(\epsilon'_i)_{i \geq l}$ belong to $\D^{\infty}$, and, $\epsilon'_l \ne 0$.
\end{lemma}

\begin{prop}
The boundary of $\E$ is the union of the $18$ non empty
regions $\E (u)$, $u\in \{ a,-a ; a\in A\}$, whose pairwise intersections have measure zero, where

$$
\E (u)= \E \cap (\E + u) \hbox{ and }
$$

$$
A=\{ 1,
1+ \alpha,
1+\alpha^2,
1+ \alpha+ \alpha^2 ,
\alpha^{-3}+ \alpha^{-2}+ 1 + \alpha^{3}, \alpha,
\alpha+\alpha^2,
\alpha^2,
\alpha^{-2}+ \alpha^{-1}+ \alpha
\} .
$$
\end{prop}
\begin{proof}
Let $u$ be an element of $A$, then $u$ is a state of the automaton $\A$.
From Lemma \ref{lemme-tec2}, there exist $(\varepsilon_i)_{i \geq 4}$ and $(\varepsilon'_i)_{i \geq 4}$ belonging to $\Dinfty$
such that $\sum_{i=4}^{\infty}\varepsilon_{i}\alpha^{i} =
u + \sum_{i=4}^{\infty}\varepsilon'_{i}\alpha^{i}$.
Thus, from Proposition \ref{pavage}, $\E \cap (\E +u)$ is not empty and with measure zero.
It will be useful to check that

$$
\alpha^{-3}+ \alpha^{-2}+ 1 + \alpha^{3} =1+ 2 \alpha+ \alpha^2 \hbox{ and }
\alpha^{-2}+ \alpha^{-1}+ \alpha = -1+ \alpha^2 .
$$

Consequently, Proposition \ref{pavage} implies $\bigcup_{u \in A} \E( u)\cup \E(-u)$ is contained in the boundary of $\E$.

 Now, let $z$ be an element of the boundary of $\E$, then by Lemma \ref{lemme-ali} there
 exist two elements of $\D^\infty,\; (\varepsilon_i)_{i \geq 4}$ and
 $(\varepsilon'_i)_{i \geq l},\;
  l \in \mathbb{Z},\; l<4$ such that $z= \sum_{i=4}^{+\infty}
 \varepsilon_i \alpha^{i}= \sum_{i=l}^{+\infty}
 \varepsilon'_i \alpha^{i}.$
We can suppose $\varepsilon_l = 1$.
Let us consider the following four cases.

\medskip

Suppose $l =-3$.
From Lemma \ref{lemme-tec1}, we deduce that $z \in \E (\alpha^{-3}+\alpha^{-2} + 1+ \alpha^3  )$.

\medskip

Suppose $l=-2$.
From Lemma \ref{lemme-tec1}, we deduce that $z \in \E (\alpha^{-2}+\alpha^{-1} + \alpha  )$.

\medskip

Suppose $l=-1$.
From Lemma \ref{lemme-tec1}, we deduce that  $z= \sum_{i=4}^{+\infty}
 \varepsilon_i \alpha^{i}= \alpha^{-1} + 1 +  \alpha^2 +\sum_{i=4}^{+\infty}
 \varepsilon'_i \alpha^{i}.$
Corollary \ref{coro-auto} implies that $t =(0,1)(0,1)(0,0)(0,1)(0,0)(\varepsilon_4 , \varepsilon'_4) \dots $ is an infinite
 path of the automaton $\A$ starting at the initial state.
 Using the automaton, we see that $t =(0,1)(0,1),(0,0)(0,1)(0,0)awww\dots$, where $a=(1,1)$ or $(0,0)$ and $w=(0,1)(1,0)(1,0)(0,1)$.
Consequently, $z = \alpha^{-1} + 1 +  \alpha^2 +\alpha^4+\alpha^5 + \sum_{i=2}^{\infty} (\alpha^{4i}+\alpha^{4i+1})$
 or $z = \alpha^{-1} + 1 +  \alpha^2 +\alpha^5 + \sum_{i=2}^{\infty} (\alpha^{4i}+\alpha^{4i+1})$.
Thus,
$z = -\alpha -\alpha^2+\alpha^6 + \sum_{i=2}^{\infty} (\alpha^{4i}+\alpha^{4i+1})$
 or $z = -1 -2\alpha - \alpha^2 + \sum_{i=1}^{\infty} (\alpha^{4i}+\alpha^{4i+1})$, and,
$z \in \E (-\alpha -\alpha^2 ) \cup  \E (-1 -  2\alpha -\alpha^2  ).$

\medskip

Suppose $l \geq0$.
Then $z= \sum_{i=4}^{+\infty}
 \varepsilon_i \alpha^{i}= \varepsilon'_0+ \varepsilon'_1 \alpha+ \varepsilon'_2
 \alpha^2 + \varepsilon'_3 \alpha^3+  \sum_{i=4}^{+\infty}
 \varepsilon'_i \alpha^{i}.$

 If $\varepsilon'_3= 0$, then $z \in \E(u)$ where $u=\varepsilon'_0+ \varepsilon'_1 \alpha+ \varepsilon'_2
 \alpha^2.$

If  $\varepsilon'_3= 1$ and $\varepsilon'_4= 0$,
  then $z= (\varepsilon'_0-1)+ (\varepsilon'_1-1) \alpha+
  (\varepsilon'_2 -1)
 \alpha^2 +  \alpha^4+  \sum_{i=5}^{+\infty}
 \varepsilon'_i \alpha^{i}
   \in \E(u)$ where $u=(\varepsilon'_0-1)+ (\varepsilon'_1-1) \alpha+
  (\varepsilon'_2 -1)
 \alpha^2 .$

 Now suppose $\varepsilon'_3= \varepsilon'_4= 1$ and $\varepsilon'_5= 0$.
 Then $\varepsilon'_1= 0$ or
  $\varepsilon'_2 =0$,
and, $z= \varepsilon'_0+ (\varepsilon'_1-1) \alpha+
  (\varepsilon'_2 -1)\alpha^2+
   \alpha^5+  \sum_{i=6}^{+\infty}
 \varepsilon'_i \alpha^{i}.$
Hence :
 \begin{itemize}
 \item
 If $\varepsilon'_0= 0,$ then $z \in \E(u)$ where $u= (\varepsilon'_1-1) \alpha+
  (\varepsilon'_2 -1)\alpha^2.$
 \item
 If $\varepsilon'_0=1$, then $t=(0,1)
 (0,\varepsilon'_1)(0,\varepsilon'_2)(0,1)(\varepsilon_4 ,1) (
 \varepsilon_5 ,0) \ldots$ is an infinite path in the automaton beginning in the
 initial state. This implies that $t= (0,1)(0,1)(0,0) (0,1)(1,1)(0,0)ww\ldots$
where $w=(0,1)(1,0)(1,0)(0,1).$ Hence
$z=
1+ \alpha+ \alpha^{3}+\alpha^{4}+\alpha^{6} +\sum_{i=2}^{\infty} (\alpha^{4i+1}+\alpha^{4i+2})$.
Thus
$z+ \alpha^{-2}+ \alpha^{-1}+\alpha= \alpha^{5}+\alpha^{6} +\sum_{i=2}^{\infty} (\alpha^{4i+1}+\alpha^{4i+2})$
and $z$ belongs to $\E (-\alpha^{-2}- \alpha^{-1}- \alpha )$.
\end{itemize}

If $\varepsilon'_3= \varepsilon'_4= \varepsilon'_5= 1,$ then
 $\varepsilon'_2=\varepsilon'_6=0.$
 Hence $z= \varepsilon'_0+ \varepsilon'_1 \alpha -\alpha^2+
   \alpha^6+  \sum_{i=7}^{+\infty}
 \varepsilon'_i \alpha^{i}.$

\begin{itemize}
\item
If $\varepsilon'_0= \varepsilon'_1=0$, then $z \in \E( -\alpha^2)$.
\item
When $\varepsilon'_0+ \varepsilon'_1  = 1$, there is no infinite path in
the automaton starting in the initial state and beginning with
$(0, \varepsilon'_0) (0,\varepsilon'_1) (0,0) (0,1) (\varepsilon_4, 1) (\varepsilon_5,1)$.
\item
Hence it remains to consider the case : $\varepsilon'_0+ \varepsilon'_1  = 2$.
But it is easy to check that this implies $\varepsilon'_6 =1$ which is not possible.
\end{itemize}

This ends the proof.
\end {proof}

Using the automaton given in the Annexe, we deduce the
following result.

\begin{prop}
Let $X =\E (1+\alpha + \alpha^2)$ and $Y = \E (1+\alpha )$. Then,

$$
\begin{array}{ll}
a) \ \E (1)= 1+ \alpha X, & b) \ \E ( \alpha^2)= - \frac{1}{\alpha }-1- \alpha + \frac{X}{\alpha}, \\
c) \ \E (1+ \alpha^2)= \left\{\frac{\alpha^4}{1- \alpha^2 }\right\}, & d) \ \E (\alpha^{-2}+ \alpha^{-1}+ \alpha )= \left\{\frac{\alpha^5 + \alpha^6}{1- \alpha^4 }\right\},\\
e)\E (\alpha^{-3}+ \alpha^{-2}+ 1
 + \alpha^{3})= \left\{\frac{\alpha^4 + \alpha^5}{1- \alpha^4 }\right\} &  f) \ \E
(\alpha )=  f_{0}(X) \cup f_{1}(X) \cup f_{1}(Y)
\end{array}
$$

$$
g) \
\E (\alpha+ \alpha^2)=  g_0 (X) \cup g_1 (X) \cup g_1 (Y) \cup g_2
(Y) \cup g_3 (Y), \hbox{ where }
$$

$$
\begin{array}{lll}
f_{0}(z)= \alpha + \alpha^2 z, & f_{1}(z)= \alpha +\alpha^4 +
\alpha^2 z, & g_0 (z) = \alpha^5 +\alpha^4 z,\\
g_1 (z) = \alpha^5 + \alpha^6+\alpha^4 z, & g_2 (z) = \alpha z, &
g_3 (z) = \alpha^4+\alpha z ,
\end{array}
$$

$$
h)\  X = \bigcup_{i=0}^{4}h_i (X)  \cup h_1 (Y) \cup h_3 (Y) \hbox{ and }
i) \ Y= \bigcup_{i=5}^{11}h_i (Y)  \cup \bigcup_{i=12}^{17}h_i (X), \hbox{ where }
$$

$$
\begin{array}{ll}
h_0 ( z )=\alpha^4 + \alpha^4 z, & h_1 (z) =\alpha^4  + \alpha^6 + \alpha^4 z \\
h_2 (z)=    \alpha^4  + \alpha^5 + \alpha^4 z, & h_3 (z)= \alpha^4 +
\alpha^5  + \alpha^6 + \alpha^4 z,\\
h_4 (z)= 1 +\alpha + \alpha^2   +\alpha^7 + \alpha^5 z, & h_5 (z)=
 h_2 (z) \\
h_6 (z) = \alpha^4 + \alpha^7 + \alpha^4 z,& h_7 (z)=
\alpha^4  + \alpha^8  + \alpha^9 + \alpha^7 z ,\\
h_8 (z)= h_0 ( z ), & h_9 (z)= \alpha^4
+\alpha^5 +\alpha^7  + \alpha^4 z,\\
h_{10} (z)= \alpha^4 + \alpha^5        + \alpha^8  + \alpha^9
 + \alpha^7 z, & h_{11}(z) = 1 +\alpha      +
\alpha^6  + \alpha^7 + \alpha^5 z \\
h_{12}(z)= \alpha^4 + \alpha^8  + \alpha^7 z, & h_{13}(z)=  h_7 (z), \\
h_{14}(z)=    \alpha^4 + \alpha^5   + \alpha^8   + \alpha^7 z, &
h_{15}(z)=   h_{10} (z)\\
h_{16} (z)= 1 +\alpha + \alpha^6 + \alpha^5 z, & h_{17}(z)= h_{11} (z),
\end{array}
$$
\end{prop}

\begin{proof}
a)
The set $1 +\alpha X$ is clearly included in $1 + \alpha \E$.
Moreover it is easy to check that $1+\alpha X$ is a subset of $\alpha^4+ \alpha \E$ which is
included in $ \E$.
Hence $1+\alpha X \subset \E (1)$.

On the other hand, let $z \in \E (1)$.
Then, there exist $(\varepsilon_i)_{i\geq 4}$ and $(\varepsilon'_i)_{i\geq 4}$ in $\D^\infty$ such that
$z=1+\sum_{i\geq 4} \varepsilon_i \alpha^i = \sum_{i\geq 4} \varepsilon'_i \alpha^i$.
From Corollary \ref{coro-auto}, $(1,0)(0,0)(0,0)(0,0)(\varepsilon_4 , \varepsilon'_4)$ $(\varepsilon_5 ,\varepsilon'_5)$
is a finite path in the automaton $\A$ starting at the initial state.
Following this path in the automaton we deduce $(\varepsilon_4 , \varepsilon'_4)=(0,1)$ and
$(\varepsilon_5 , \varepsilon'_5)=(1,0)$.
It gives $z= 1+ \alpha^5+ \alpha^2 w= \alpha^4+ \alpha^2 w'$ where $w, w' \in \E$.
Consequently
 $ \E(1) \subset (1 + \alpha \E) \cap ( \alpha^4+ \alpha \E)=1+\alpha (\E \cap (1+\alpha +\alpha^2+\E)) = 1 +\alpha X$.

b)
We have
$
1 + \alpha + \alpha^2 + \alpha \E(\alpha^2)= (\alpha^4+  \alpha \E) \cap (1+ \alpha+ \alpha^2 + \alpha \E)
\subset
\E \cap (1+ \alpha+ \alpha^2 + \E)
= X$.
Hence
$ \E (\alpha^2) \subset - \frac{1}{\alpha }-1- \alpha + \frac{X}{\alpha}.$
To prove the other inclusion, let
$z \in X$.
Then by the automaton we deduce that
$z=
 1+ \alpha + \alpha^2  + \alpha w= \alpha^4 + \alpha w',\; w, w' \in \E. $
Hence $- \frac{1}{\alpha }-1- \alpha + \frac{z}{\alpha}= w= \alpha^2  + w'  $
and $- \frac{1}{\alpha }-1- \alpha + \frac{X}{\alpha}  \subset \E (\alpha^2)$.

c)
Let $z \in \E (1 + \alpha^2)$ : $z=\sum_{i\geq 4} \varepsilon_i \alpha^i = 1+\alpha^2 +\sum_{i\geq 4} \varepsilon'_i \alpha^i$,
$(\varepsilon_i)_{i\geq 4}$, $(\varepsilon'_i)_{i\geq 4} \in \D^\infty$.
Corollary \ref{coro-auto} and the automaton show that $(0,1)(0,0)(0,1)(0,0)(\varepsilon_4 ,\varepsilon'_4) \dots$ is a infinite path
starting in the initial state and $(\varepsilon_i , \varepsilon_i')_{i\geq 4}$ is equal to  $uuu\dots$ where $u=(1,0)(0,1)$.
Then,
$z= \alpha^4 + \alpha^6  +\alpha^8 + \cdots =
 1+ \alpha^2 + \alpha^5  + \alpha^7 +\alpha^9 + \dots$
and
$\E (1 + \alpha^2)=
\alpha^4 ( 1- \alpha^2 )^{-1} $.

d)
Let $z \in \E ( \alpha^{-2}+ \alpha^{-1}+ \alpha)$.
From Corollary \ref{coro-auto} and using the automaton we deduce that
$z=\alpha^{-2}+ \alpha^{-1}+ \alpha  +\alpha^{4}+\sum_{i=2}^{\infty} (\alpha^{4i-1}+\alpha^{4i})  = \sum_{i=1}^{\infty} (\alpha^{4i+1}+\alpha^{4i+2}).$
Hence
$\E (\alpha^{-2}+ \alpha^{-1}+ \alpha)= ( \alpha^5 + \alpha^6 )(1- \alpha^4)^{-1}$.

e)
Corollary \ref{coro-auto} and the automaton give the result.

f)
Let $z \in \E(\alpha)$.
Then, there exist $(\varepsilon_i)_{i\geq 4}$ and $(\varepsilon'_i)_{i\geq 4}$ in $\D^\infty$ such that
$z=\sum_{i\geq 4} \varepsilon_i \alpha^i = \alpha + \sum_{i\geq 4} \varepsilon'_i \alpha^i$.
From Corollary \ref{coro-auto},
$(0,0)(0,1)(0,0)(0,0)(\varepsilon_4 , \varepsilon'_4 )\dots $ is a path in the automaton starting in the initial state.
Hence, $(\varepsilon_4 , \varepsilon'_4 )(\varepsilon_5 , \varepsilon'_5 )(\varepsilon_6, \varepsilon'_6 )$ belongs to
$\{ (0,0), (1,1),(0,1)\}(1,0)(0,1)$.
Consequently, $z$ belongs to the union of
$(\alpha^5 + \alpha^2 \E ) \cap (\alpha + \alpha^2 \E) $,
$(\alpha^4+ \alpha^5 + \alpha^2 \E) \cap (\alpha + \alpha^4+ \alpha^2 \E) $
and
$( \alpha^5 + \alpha^2 \E) \cap (\alpha + \alpha^4+ \alpha^2 \E) $ which is equal to $f_{0}(X) \cup f_{1}(X) \cup f_{1}(Y)$.
Hence $ \E (\alpha )=  f_{0}(X) \cup f_{1}(X) \cup f_{1}(Y)$.

g)
Let $z \in \E(\alpha  + \alpha^2)$.
Then, there exist $(\varepsilon_i)_{i\geq 4}$ and $(\varepsilon'_i)_{i\geq 4}$ in $\D^\infty$ such that
$z=\sum_{i\geq 4} \varepsilon_i \alpha^i = \alpha + \alpha^2 + \sum_{i\geq 4} \varepsilon'_i \alpha^i$.
From Corollary \ref{coro-auto},
$(0,0)(0,1)(0,1)(0,0)(\varepsilon_4 , \varepsilon'_4 )\dots $ is a path in the automaton starting in the initial state.
Hence, we either have
\begin{enumerate}
\item
$((\varepsilon_i , \varepsilon'_i ))_{4\leq i\leq 7} \in (0,1)(1,0) \{ (0,0), (1,1),(1,0)\}(0,1)$,
\item
$(\varepsilon_4 , \varepsilon'_4 )\in \{ (0,0),(1,1) \}$, or
\item
$((\varepsilon_i , \varepsilon'_i ))_{i\geq 4} \in (0,1) \{ (0,0)(0,0), (0,0)(1,1),(1,1)(0,0)\} ww \dots $,
\end{enumerate}

where $w=(0,1)(1,0)(1,0)(0,1)$.
This means that $z$ belongs to

$$
\begin{array}{lll}
& & \left(
(\alpha^5 + \alpha^4 \E)
\cap
(\alpha + \alpha^2  +\alpha^4 + \alpha^7+  \alpha^4 \E )
\right) \\
& \cup &
\left(
(\alpha^5 + \alpha^6 +\alpha^4 \E)
\cap
(\alpha + \alpha^2  +\alpha^4 +\alpha^6+ \alpha^7+  \alpha^4 \E )
\right) \\
&\cup &
\left(
(\alpha^5 + \alpha^6 +\alpha^4 \E)
\cap
(\alpha + \alpha^2  +\alpha^4 + \alpha^7+  \alpha^4 \E )
\right) \\
&\cup &
\left(
(\alpha \E \cap (\alpha + \alpha^2 +  \alpha \E )
\right)
\cup
\left((\alpha^4 + \alpha \E) \cap (\alpha + \alpha^2  +\alpha^4 +   \alpha \E ) \right)\\
&\cup  & \left\{ z_1 , z_2 , z_3 \right\} \\
= &  & g_0 (X)  \cup g_1 (X) \cup g_1 (Y) \cup g_2
(Y) \cup g_3 (Y)
\cup
\left\{ z_1 , z_2 , z_3 \right\}
\end{array}
$$

where
$
z_1
=
\sum_{i=2}^{+\infty} (\alpha^{4i}+  \alpha^{4i+1})=
\alpha+  \alpha^2 + \alpha^4 +\alpha^7+ \sum_{i=2}^{+\infty}
(\alpha^{4i+2}+  \alpha^{4i+3}),
 z_2
=
\alpha^6 +z_1 \hbox{ and }
z_3
=
\alpha^5 +z_1.
$
We can also check that
$(1,0)(1,0)(0,0)(0,0)uuu\dots $, where $u=(0,1)(1,1)(1,0)(1,0)$, is an infinite path of the automaton starting in the initial state.
Consequently,
$
z_1= \alpha^4+ \alpha^5 + \sum_{i=2}^{+\infty} (\alpha^{4i+1}+
\alpha^{4i+2})
$
and

$$
z_1 \in \left((\alpha^4 + \alpha \E) \cap
(\alpha + \alpha^2  +\alpha^4 +   \alpha \E ) \right)= g_3(Y).
$$

Moreover, it shows that $z_2$ belongs to $g_3(Y)$.
In the same way,
$
z_1=
\alpha^5+  \alpha^6 + \alpha^8 + \sum_{i=2}^{+\infty}
(\alpha^{4i+3}+ \alpha^{4i+4})
$.
Thus, $z_3 = 2 \alpha^5+  \alpha^6
+ \alpha^8 + \sum_{i=2}^{+\infty} (\alpha^{4i+3}+ \alpha^{4i+4})=
\alpha^5+  \sum_{i=2}^{+\infty} (\alpha^{4i}+  \alpha^{4i+1}) .$
But $2 \alpha^5 + \alpha^6 = \alpha+ \alpha^2 + \alpha^7$, consequently
 $z_3$ belongs to $\left( \alpha \E) \cap (\alpha + \alpha^2 +
\alpha \E ) \right)= g_{2}(Y)$.

h)
Let $z \in X= \E(1+ \alpha  + \alpha^2)$.
Then, there exist $(\varepsilon_i)_{i\geq 4}$ and $(\varepsilon'_i)_{i\geq 4}$ in $\D^\infty$ such that
$z=\sum_{i\geq 4} \varepsilon_i \alpha^i = 1+ \alpha + \alpha^2 + \sum_{i\geq 4} \varepsilon'_i \alpha^i$.
From Corollary \ref{coro-auto},
we necessarily have $(\varepsilon_4 , \varepsilon'_4)=(1,0)$ and one of the following situations :

\begin{enumerate}
\item
$((\varepsilon_i , \varepsilon'_i ))_{ i\geq 5} \in (1,0)\{ (0,0), (1,1)\} ww\dots $ where $w =(0,1)(1,0)$;
\item
$((\varepsilon_i , \varepsilon'_i ))_{i\geq 5} \in (0,1) \{ (0,0),(1,1)\} ww\dots  $ where $w=(0,1)(1,0)(1,0)(0,1)$;
\item
$(\varepsilon_i , \varepsilon'_i )_{ 5\leq i\leq 8} \in \{ (0,0),(1,1)\}^2 (0,1)(1,0)$;
\item
$(\varepsilon_i , \varepsilon'_i )_{ 5\leq i\leq 9} = (1,0)(1,0)(0,1)(1,0)(0,1)$;
\item
$(\varepsilon_i , \varepsilon'_i )_{ 5\leq i\leq 8} \in \{ (0,0),(1,1)\} (1,0)(0,1)(1,0)$.
\end{enumerate}

This means $z$ belongs to $\bigcup_{i=0}^{4}h_i (X)  \cup h_1 (Y)
\cup h_3 (Y) \cup \{x_1, x_2,x_3, x_4\}$ where

\begin{align*}
x_{1}    = & \alpha^4 + \alpha^5 + \sum_{i=4}^{+\infty} \alpha^{2i}= 1+\alpha+ \alpha^2 + \sum_{i=3}^{+\infty} \alpha^{2i+1},\\
x_{2}    = &  x_1 + \alpha^6  ,\\
x_{3}    = & \alpha^4 +  \sum_{i=2}^{+\infty} (\alpha^{4i}+ \alpha^{4i+1})= 1+\alpha+ \alpha^2 + \alpha^5+ \alpha^7+ \sum_{i=2}^{+\infty} (\alpha^{4i+2}+ \alpha^{4i+3}),\\
x_{4}    = & x_3 +\alpha^6,\\
h_{0}(X) = & (\alpha^4 + \alpha^4 \E) \cap (1+\alpha + \alpha^2  + \alpha^7+  \alpha^4 \E ), \\
h_{1}(X) = & (\alpha^4+\alpha^6  + \alpha^4 \E) \cap (1+\alpha + \alpha^2  +\alpha^6+ \alpha^7+  \alpha^4 \E ), \\
h_{2}(X) = & (\alpha^4+\alpha^5 + \alpha^4 \E) \cap (1+\alpha + \alpha^2  +\alpha^5+ \alpha^7+  \alpha^4 \E ), \\
h_{3}(X) = & (\alpha^4+\alpha^5 +\alpha^6+ \alpha^4 \E) \cap (1+\alpha + \alpha^2  +\alpha^5+\alpha^6+ \alpha^7+  \alpha^4 \E ), \\
h_{4}(X) = & (\alpha^4+\alpha^5 +\alpha^6 +\alpha^8 + \alpha^5 \E) \cap (1+\alpha + \alpha^2 +  \alpha^7+  \alpha^5 \E ) , \\
h_{1}(Y) = & (\alpha^4+\alpha^6 + \alpha^4 \E) \cap (1+\alpha + \alpha^2  + \alpha^7+  \alpha^4 \E ) \hbox{ and } \\
h_{3}(Y) = &  (\alpha^4+ \alpha^5 + \alpha^6 + \alpha^4 \E) \cap (1+\alpha +\alpha^2  + \alpha^5 + \alpha^7+  \alpha^4 \E ) .\\
\end{align*}

We easily can check (using Corollary \ref{coro-auto} and the automaton) that

$$
x_1 = x_3 =  \alpha^4 + \sum_{i=2}^{+\infty}(\alpha^{4i}+ \alpha^{4i+1}) ,
$$

and thus $x_1 \in  h_{0}(X)$, $x_2 \in h_{1}(X)$, and $x_2 = x_4$, which concludes the proof of h).

i)
Let $z \in X= \E(1+ \alpha )$. Then, there exist
$(\varepsilon_i)_{i\geq 4}$ and $(\varepsilon'_i)_{i\geq 4}$ in
$\D^\infty$ such that $z=\sum_{i\geq 4} \varepsilon_i \alpha^i = 1+
\alpha + \sum_{i\geq 4} \varepsilon'_i \alpha^i$.
From Corollary \ref{coro-auto},
we necessarily have $(\varepsilon_4 , \varepsilon'_4)=(1,0)$ and one
of the following situations :

\begin{enumerate}
\item
$((\varepsilon_i , \varepsilon'_i ))_{ i\geq 5} \in \{ (0,0),
(1,1)\} (0,1)^2  \{ (0,0), (1,1)\}^2 ww\dots $ ;
\item
$((\varepsilon_i , \varepsilon'_i ))_{ i\geq 5} \in (1,0)\{ (0,0),
(1,1)\}^2 (1,0)(0,1) ww\dots $;
\item
$((\varepsilon_i , \varepsilon'_i ))_{ 5\leq i\leq 7} \in \{ (0,0),
(1,1)\} (0,1) \{ (0,0), (1,1)\}$;
\item
$((\varepsilon_i , \varepsilon'_i ))_{ 5\leq i\leq 8} \in \{ (0,0),
(1,1)\} (0,1)^2 (1,0)\{ (0,0), (1,1), (1,0)\} (0,1)(1,0)$;
\item
$((\varepsilon_i , \varepsilon'_i ))_{5\leq i\geq 8} \in (1,0)(0,1)\{
(0,0), (1,1), (0,1)\} (1,0)(0,1)$.

\end{enumerate}

where $w=(0,1)(1,0)(1,0)(0,1)$.
Hence $z$ belongs to

$$
\left(\bigcup_{i=5}^{11}h_i (Y)\right)
\cup
\left(\bigcup_{i=12}^{17}h_i (X) \right)
\cup \{y_i ; 1\leq i\leq 8\} ,
$$

where

\begin{align*}
y_1= &  \alpha^4 + \sum_{i=2}^{+\infty} (\alpha^{4i+3}+
\alpha^{4i+4})= 1+ \alpha+\alpha^6 + \alpha^7 +\alpha^{10} + \sum_{i=3}^{+\infty} (\alpha^{4i+1}+
\alpha^{4i+2}),\\
y_2= & y_1+ \alpha^9,\; y_3= y_1+ \alpha^8 ,\; y_4= y_1+
\alpha^5,\; y_5= y_1+ \alpha^5 + \alpha^9.\\
y_6= & y_1 + \alpha^5+\alpha^8= 1+ \alpha+ \sum_{i=2}^{+\infty}
(\alpha^{4i+1}+ \alpha^{4i+2}), \;  y_7=  y_6+ \alpha^7,\; y_8= y_6+ \alpha^6 ,\\
h_{5}(Y) = & (\alpha^4+\alpha^5  +  \alpha^4 \E) \cap (1+\alpha + \alpha^5 +\alpha^6 +      \alpha^4 \E ) ,\\
h_{6}(Y) = &  (\alpha^4+\alpha^7  + \alpha^4 \E) \cap (1+\alpha   +\alpha^6+ \alpha^7+  \alpha^4 \E ),\\
h_{7}(Y) = &  (\alpha^4+\alpha^8 +\alpha^9+ \alpha^7 \E) \cap (1+\alpha + \alpha^6  +\alpha^7+ \alpha^{10}+  \alpha^7 \E ), \\
h_{8}(Y) = &  (\alpha^4+  \alpha^4 \E) \cap (1+\alpha + \alpha^6 +    \alpha^4 \E ), \\
h_{9}(Y) = &  (\alpha^4+\alpha^5 +\alpha^7 +  \alpha^4 \E) \cap (1+\alpha + \alpha^5 +\alpha^6 +    \alpha^7+  \alpha^4 \E ),\\
h_{10}(Y)= &  (\alpha^4+\alpha^5+\alpha^8+\alpha^9 + \alpha^7 \E) \cap (1+\alpha + \alpha^5  + \alpha^6 + \alpha^7  + \alpha^{10} + \alpha^7 \E ),\\
h_{11}(Y)= &  (\alpha^4+\alpha^5+\alpha^8+ \alpha^5 \E) \cap (1+\alpha +  \alpha^6 + \alpha^7  +   \alpha^5 \E ) ,\\
h_{12}(Y)= &  (\alpha^4+ \alpha^8+ \alpha^7 \E) \cap (1+\alpha   + \alpha^6 + \alpha^7  + \alpha^{10}  +\alpha^7 \E ), \\
h_{13}(X)= &  (\alpha^4+\alpha^8 +\alpha^9 +  \alpha^7\E) \cap (1+\alpha + \alpha^6 +\alpha^7 +    \alpha^9+\alpha^{10}+  \alpha^7 \E ), \\
h_{14}(X)= &  (\alpha^4+\alpha^5+ \alpha^8+ \alpha^7 \E) \cap (1+\alpha + \alpha^5  + \alpha^6 + \alpha^7  + \alpha^{10}  +\alpha^7 \E ), \\
h_{15}(X)= &  (\alpha^4+\alpha^5+ \alpha^8+\alpha^9+ \alpha^7 \E) \cap (1+\alpha + \alpha^5  + \alpha^6 + \alpha^7 +\alpha^9 + \alpha^{10}  +\alpha^7 \E ) ,\\
h_{16}(X)= &  (\alpha^4+\alpha^5+\alpha^8+ \alpha^5 \E) \cap (1+\alpha +  \alpha^6 + \alpha^5 \E ) ,\\
h_{17}(X)= &  (\alpha^4+\alpha^5+\alpha^7 +\alpha^8+ \alpha^5 \E) \cap (1+\alpha +  \alpha^6 +\alpha^7 + \alpha^5 \E )  ,\\
\end{align*}

Let us prove that for each integer $i \in \{1,\ldots, 8\}$, there
exists $j \in  \{5,\ldots, 11\}$ or $k \in  \{12,\ldots, 17\}$ such
that $y_i$ belongs to  $h_{j}(X)$ or to $h_{k}(Y).$

Indeed, since $y_1= \alpha^4+ \alpha^3 z_1$ (see case g)), then $y_1 \in
(\alpha^4+\alpha^7  + \alpha^4 \E) \cap ( 2 \alpha^4+\alpha^5
+\alpha^7 + \alpha^4 \E)= h_{6}(Y).$
We deduce that $y_2$ and $y_3$ belong also to $h_{6}(Y)$, and, $y_4$
and $y_5$ belong to $h_{9}(Y)$.

Using the automaton we can verify that $y_6 = 1+ \alpha + \alpha^5 + \alpha^6+ \sum_{i=2}^{+\infty } (\alpha^{4i+2}  + \alpha^{4i+3})$.
Hence, $y_6$ belongs to $1+ \alpha + \alpha^5 + \alpha^6+ \alpha^4 \E$.
But it also belongs to $\alpha^4 +\alpha^5 +\alpha^4\E $.
Thus $y_6 \in h_{5}(Y)$ and $y_7 \in h_9 (Y)$.

We have  $y_8 = y_6+ \alpha^6 \in (1+\alpha + \alpha^6+ \alpha^5 \E
)$.
On the other hand we can check using the automaton that $y_8=
\alpha^4+ \alpha^5 +  \alpha^8 +\sum_{i=5}^{+\infty} \alpha^{2i}$,
hence $y_8 \in (\alpha^4 +\alpha^5+ \alpha^8+  \alpha^5 \E )$ and
$y_8$ belongs to $h_{16}(X)$.
\end{proof}

\vspace{1em}
{\bf Remarks and comments.}
There are points which has at least $6$ expansions in base $\alpha$. For example:

$$
\begin{array}{lll}
\alpha+  \sum_{i=2}^{+\infty } \alpha^{2i} & = \sum_{i=1}^{+\infty} (\alpha^{4i}+ \alpha^{4i+1})\\
 & = 1 + \alpha+ \alpha^2+ \sum_{i=2}^{\infty} \alpha^{2i+1} \\
 & = 1 + \alpha+\sum_{i=1}^{\infty} (\alpha^{4i+1}+ \alpha^{4i+2}))\\
 & = \alpha + \alpha^2+\alpha^4+\sum_{i=1}^{\infty} (\alpha^{4i+3}+ \alpha^{4i+4})\\
 & = \alpha^{-3} + \alpha^{-2}+1+ \alpha^3+\sum_{i=1}^{\infty} (\alpha^{4i+2}+ \alpha^{4i+3}) .
\end{array}
$$

We address the two following questions :

\begin{enumerate}
\item
Can you parameterize the boundary of $\E_{1,1,1}$ ?
\item
Does this boundary be homeomorphic to the sphere ?
\end{enumerate}

The technics used in this work can be used to study $\E_{a_1,a_2,\dots , a_d}$ with the assumption that $a_1\geq a_2\geq \dots \geq a_d\geq 1$.

\bigskip

{\bf acknowledgements.}
Both authors would like to thank the Brazil-France agreement for cooperation in Mathematics and the CNRS that permitted the authors to visit each other.
This strongly contributed to this work.
The first author would like to thank the hospitality of the Mathematic department of the State University of S\~ao Paolo in S\~ao Jos\'e do Rio Preto (Brazil).
The second author thanks the Laboratoire Ami\'enois de Math\'ematiques Fondamentales et Appliqu\'ees, CNRS-UMR 6140, from the University of Picardie Jules Verne (France) for his hospitality.

The second author was supported by a CNPq grant Proc. 305043/2006-4.

\newpage

{\bf Annexe.}

\vskip 3cm

\hskip 1cm

\MinusPicture{\VCDraw
{
\begin{VCPicture}{(-15,-23)(15,23)}
%
\SwivelLabel


\VCPut[0]{(0,23)}{\StateVar[0]{(0,0)} {0000000}}
\VCPut[0]{(-10,23)}{\StateVar[-\alpha-\alpha^2-\alpha^3]{(0,0)} {m0000111}}
\VCPut[0]{(10,23)}{\StateVar[\alpha+\alpha^2+\alpha^3]{(0,0)}    {0000111}}

\VCPut[0]{(-14,20)}{\StateVar[-\alpha^3]{(0,0)}                 {m0000001}}
\VCPut[0]{(14,20)}{\StateVar[\alpha^3]{(0,0)}                    {0000001}}

\VCPut[0]{(-4,16)}{\StateVar[-1-\alpha]{(0,0)} {m0001100}}
\VCPut[0]{(4,16)}{\StateVar[1+\alpha]{(0,0)}    {0001100}}

\VCPut[0]{(-10,12)}{\StateVar[-\alpha^2-\alpha^3]{(0,0)} {m0000011}}
\VCPut[0]{(-6,12)}{\StateVar[-\alpha-\alpha^2]{(0,0)}    {m0000110}}
\VCPut[0]{(10,12)}{\StateVar[\alpha^2+\alpha^3]{(0,0)}    {0000011}}
\VCPut[0]{(6,12)}{\StateVar[\alpha+\alpha^2]{(0,0)}       {0000110}}


\VCPut[0]{(-6,8)}{\StateVar[-1-\alpha^2]{(0,0)}       {m0001010}}
\VCPut[0]{(6,8)}{\StateVar[1+\alpha^2]{(0,0)}          {0001010}}

\VCPut[0]{(-15,4)}{\StateVar[1+\alpha+\alpha^2]{(0,0)} {0001110}}
\VCPut[0]{(-11,4)}{\StateVar[-\alpha^2]{(0,0)}        {m0000010}}
\VCPut[0]{(-6,4)}{\StateVar[-\alpha-\alpha^3]{(0,0)}  {m0000101}}
\VCPut[0]{(6,4)}{\StateVar[\alpha+\alpha^3]{(0,0)}     {0000101}}
\VCPut[0]{(11,4)}{\StateVar[\alpha^2]{(0,0)}           {0000010}}
\VCPut[0]{(15,4)}{\StateVar[-1-\alpha-\alpha^2]{(0,0)}{m0001110}}

\VCPut[0]{(-4,-4)}{\StateVar[1+\alpha+\alpha^3]{(0,0)}  {0001101}}
\VCPut[0]{(4,-4)}{\StateVar[-1-\alpha-\alpha^3]{(0,0)} {m0001101}}
\VCPut[0]{(-11,-4)}{\StateVar[-\alpha]{(0,0)}                    {m0000100}}
\VCPut[0]{(11,-4)}{\StateVar[\alpha]{(0,0)}                       {0000100}}
\VCPut[0]{(-15,-4)}{\StateVar[-1]{(0,0)}                         {m0001000}}
\VCPut[0]{(15,-4)}{\StateVar[1]{(0,0)}                           {0001000}}

\VCPut[0]{(-6,-8)}{\StateVar[\frac{1}{\alpha}+1+\alpha^2]{(0,0)} {0011010}}
\VCPut[0]{(6,-8)}{\StateVar[-\frac{1}{\alpha}-1-\alpha^2]{(0,0)} {m0011010}}

\VCPut[0]{(-6,-12)}{\StateVar[\frac{1}{\alpha^2}+\frac{1}{\alpha} +\alpha]{(0,0)}  {0110100}}
\VCPut[0]{(6,-12)}{\StateVar[-\frac{1}{\alpha^2}-\frac{1}{\alpha} -\alpha]{(0,0)} {m0110100}}
\VCPut[0]{(-11,-12)}{\StateVar[-1-\alpha^2-\alpha^3]{(0,0)}                            {m0001011}}
\VCPut[0]{(11,-12)}{\StateVar[1+\alpha^2+\alpha^3]{(0,0)}                               {0001011}}

\VCPut[0]{(-6,-16)}{\StateVar[\frac{1}{\alpha^3}+\frac{1}{\alpha^2}+1+\alpha^3]{(0,0)}  {1101001}}
\VCPut[0]{(6,-16)}{\StateVar[-\frac{1}{\alpha^3}-\frac{1}{\alpha^2}-1-\alpha^3]{(0,0)} {m1101001}}




\VCPut[0]{(-2,-20)}{\StateVar[-1-\alpha^3]{(0,0)}                                    {m0001001}}
\VCPut[0]{(2,-20)}{\StateVar[1+\alpha^3]{(0,0)}                                       {0001001}}


\LoopN{0000000}{(0,0),(1,1)}
\EdgeR{0000000}{m0000001}{(0,1)}
\EdgeL{0000000}{0000001}{(1,0)}
\VArcR[.4]{arcangle=-80,ncurv=1.3}{0000111}{0001110}{(0,0),(1,1)}
\VArcL[.4]{arcangle=80,ncurv=1.3}{m0000111}{m0001110}{(0,0),(1,1)}


\EdgeL[0.1]{m0001100}{0000110}{(1,0)}
\EdgeR[0.1]{0001100}{m0000110}{(0,1)}

\EdgeL{m0000110}{m0001100}{(0,0),(1,1)}
\EdgeR[0.1]{m0000110}{m0001101}{(0,1)}

\EdgeR{0000110}{0001100}{(0,0),(1,1)}
\EdgeL[0.1]{0000110}{0001101}{(1,0)}

\EdgeR{m0000011}{m0000111}{(0,1)}
\EdgeL{m0000011}{m0000110}{(0,0)}
\EdgeR{m0000011}{m0000110}{(1,1)}

\EdgeL{0000011}{0000111}{(0,1)}
\EdgeL{0000011}{0000110}{(1,1)}
\EdgeR{0000011}{0000110}{(0,0)}

\EdgeL[0.2]{m0000001}{m0000011}{(0,1)}
\EdgeR{m0000001}{m0000010}{(1,1),(0,0)}

\EdgeR[0.2]{0000001}{0000011}{(0,1)}
\EdgeL{0000001}{0000010}{(1,1),(0,0)}

\ArcL[0.5]{m0001010}{0001010}{(1,0)}
\ArcL[0.5]{0001010}{m0001010}{(0,1)}

\EdgeR{m0001110}{0000010}{(1,0)}

\EdgeL{0001110}{m0000010}{(0,1)}

\EdgeL{m0000010}{m0000101}{(0,1)}
\EdgeR[0.2]{m0000010}{m0000100}{(0,0)}
\EdgeR[0.4]{m0000010}{m0000100}{(1,1)}
\EdgeL[0.2]{m0000010}{0011010}{(1,0)}

\EdgeR{0000010}{0000101}{(1,0)}
\EdgeL[0.2]{0000010}{0000100}{(0,0)}
\EdgeL[0.4]{0000010}{0000100}{(1,1)}
\EdgeR[0.2]{0000010}{m0011010}{(0,1)}

\EdgeL[0.05]{m0000101}{m0001011}{(0,1)}
\EdgeR{m0000101}{m0001010}{(0,0)}
\EdgeL{m0000101}{m0001010}{(1,1)}

\EdgeR[0.05]{0000101}{0001011}{(1,0)}
\EdgeL{0000101}{0001010}{(0,0)}
\EdgeR{0000101}{0001010}{(1,1)}

\EdgeR[0.15]{0001101}{m0000100}{(0,1)}
\EdgeL[0.15]{m0001101}{0000100}{(1,0)}
\EdgeL[0.1]{0001101}{0011010}{(0,0)}
\EdgeL[0.3]{0001101}{0011010}{(1,1)}
\EdgeR[0.1]{m0001101}{m0011010}{(0,0)}
\EdgeR[0.3]{m0001101}{m0011010}{(1,1)}

\VArcR[.1]{arcangle=-60,ncurv=1.7}{m0000100}{m0001001}{(0,1)}
\VArcL[.1]{arcangle=60,ncurv=1.7}{0000100}{0001001}{(1,0)}
\EdgeR{m0000100}{m0001000}{(0,0)}
\EdgeL{m0000100}{m0001000}{(1,1)}
\EdgeL{0000100}{0001000}{(0,0)}
\EdgeR{0000100}{0001000}{(1,1)}

\EdgeL{0011010}{0110100}{(0,0)}
\EdgeR{0011010}{0110100}{(1,1)}
\EdgeL{m0011010}{m0110100}{(0,0)}
\EdgeR{m0011010}{m0110100}{(1,1)}

\EdgeL{0110100}{1101001}{(1,0)}
\EdgeR{m0110100}{m1101001}{(0,1)}

\ArcR[0.06]{1101001}{m0110100}{(0,1)}
\ArcL[0.06]{m1101001}{0110100}{(1,0)}


\EdgeL{m0001000}{0001110}{(1,0)}
\EdgeR{0001000}{m0001110}{(0,1)}

\EdgeL[0.03]{m0001001}{0001100}{(1,0)}
\EdgeR[0.03]{0001001}{m0001100}{(0,1)}

\VArcR[.05]{arcangle=-90,ncurv=1.25}{m0001011}{0001000}{(1,0)}
\VArcL[.05]{arcangle=90,ncurv=1.25}{0001011}{m0001000}{(0,1)}

\end{VCPicture}
}

\end{document}